\documentclass[13ptreqno]{amsart}
\usepackage{amssymb}
\usepackage{amsfonts} 
\usepackage{latexsym}   
\usepackage{amssymb}    
\usepackage{amsmath}    
\usepackage{amsbsy}
\usepackage{amsgen}
\usepackage{amsfonts}
\usepackage{array}
\usepackage{epsfig}
\usepackage{psfrag}
\usepackage[all]{xy}
\usepackage{hyperref}

\usepackage{xcolor}
\usepackage{amssymb}
\usepackage{mathabx}
\usepackage{amsfonts}
\usepackage{hyperref}
\usepackage{amsmath}
\usepackage{graphicx}
\usepackage[all]{xy}%
\providecommand{\U}[1]{\protect\rule{.1in}{.1in}}

\theoremstyle{plain}

\newtheorem{algorithm}{Algorithm}[section]
\newtheorem{thm}{Thm}

\newtheorem{corollary}[algorithm]{Corollary}

\newtheorem{definition}[algorithm]{Definition}

\newtheorem{lemma}[algorithm]{Lemma}

\newtheorem{theorem} [algorithm] {Theorem}
\newtheorem{theoremlet}[thm]{Theorem}
\newtheorem{theoremlet'}[thm]{Theorem$'$}

\usepackage{amssymb}
\usepackage{amsfonts} 
\usepackage{latexsym}   
\usepackage{amssymb}    
\usepackage{amsmath}    
\usepackage{amsbsy}
\usepackage{amsgen}
\usepackage{amsfonts}
\usepackage{array}
\usepackage{scalerel}

\usepackage{epsfig}
\usepackage{psfrag}
\usepackage[all]{xy}
\usepackage{amssymb}
\usepackage{mathabx}
\newtheorem{remark}[algorithm]{Remark}
\newtheorem{proposition}[algorithm]{Proposition}

\newtheorem*{msrconj}{Maximal Symmetry Rank Conjecture}

\newtheorem*{observe*}{Observation}

\newtheorem*{CondC}{Property C}

\def\qqq{\mathbb{Q}}
\def\rrr{\mathbb{R}}
\def\ccc{\mathbb{C}}
\def\zzz{\mathbb{Z}}

\DeclareMathOperator{\Isom}{Isom}
\DeclareMathOperator{\Fix}{Fix}
\def\codim{\textrm{codim}}
\def\bdm{\begin{displaymath}}
\def\edm{\end{displaymath}}
\def\beq{\begin{equation}}
\def\eeq{\end{equation}}
\def\bes{\begin{equation*}}
\def\ees{\end{equation*}}
\def\epcm{\end{picture}\end{center}\end{minipage}}
\def\bpcm{\begin{minipage}{80pt}\begin{center}\begin{picture}}

\def\t2{T^2}

\def\f4{F_4}
\def\g2{G_2}

\def\p2{\frac{\pi}{2}}

\def\Fix{\textrm{Fix}}

\def\dim{\textrm{dim}}

 \numberwithin{equation}{section}
  \numberwithin{figure}{section}

\newtheorem*{ack}{Acknowledgements}

\begin{document}
\newcommand{\comment}[1]{\vspace{5 mm}\par \noindent
\marginpar{\textsc{Note}}
\framebox{\begin{minipage}[c]{0.95 \textwidth}
#1 \end{minipage}}\vspace{5 mm}\par}

\title[Almost Isotropy-Maximal Manifolds of Non-negative Curvature]{Almost Isotropy-Maximal Manifolds of Non-negative Curvature}

\author[Dong]{Zheting Dong}
\address[Dong]{Huawei's Hong Kong Research Centre, Hong Kong}
\email{zhetingdong1031@gmail.com}

\author[Escher]{Christine Escher}
\address[Escher]{Department of Mathematics, Oregon State University, Corvallis, Oregon}
\email{escherc@oregonstate.edu}

\author[Searle]{Catherine Searle}
\address[Searle]{Department of Mathematics, Statistics, and Physics, Wichita State University, Wichita, Kansas}
\email{Catherine.Searle@wichita.edu}


\date{\today}

%

\begin{abstract} 

We extend the equivariant classification results of Escher and Searle \cite{ES2},  for closed, simply connected,  Riemannian $n$-manifolds with non-negative sectional curvature admitting  isometric isotropy-maximal torus actions to the class of such manifolds admitting isometric strictly almost isotropy-maximal torus actions.  In particular, we 
 prove that any such manifold is equivariantly diffeomorphic to the free, linear quotient by a torus of  a product of spheres of dimensions greater than or equal to three.

\end{abstract}
\maketitle


\section{Introduction}

The classification of closed, simply-connected Riemannian manifolds with positive or non-negative sectional curvature is a long-standing 
problem in Riemannian geometry.  One successful approach to this problem has been the Grove Symmetry Program, which aims to classify such manifolds with ``large" symmetries. 

An important first step is to consider the case of continuous abelian symmetries, that is, of torus actions.  
The Symmetry Program suggests starting with the largest possible torus action this class of manifolds can admit and then successively reducing the rank of the torus in the hopes of finding either new examples, new constructions, or new obstructions. 

In particular,  {\em isotropy-maximal} torus actions, that is, torus actions for which the rank of the largest isotropy subgroup equals the cohomogeneity of the action, are especially appealing due to the structure such an action imposes on the manifold.
Recently, Escher and Searle \cite{ES2} obtained a classification up to equivariant diffeomorphism of closed, simply connected, non-negatively curved Riemannian $n$-manifolds admitting effective,  isometric isotropy-maximal $T^{k}$-actions for $  \lfloor(n+1)/2\rfloor 
\leq k\leq \lfloor 2n/3\rfloor$.  This work suggests a different approach to the general classification problem for this class of manifolds: rather than try to classify them via the rank of the group action, instead classify such manifolds via the rank of the largest possible isotropy group of the action.  The classification in  \cite{ES2} handles the first step, namely, the case where the action is isotropy-maximal.  The next step in this program is to reduce the rank of the largest possible isotropy subgroup.

In this paper, we focus on  
{\em almost isotropy-maximal} torus actions, that is, torus actions for which there exists an isotropy subgroup of rank equal to  the cohomogeneity of the action minus one.   In particular, we say the action is {\em strictly} almost isotropy-maximal when the action is almost isotropy-maximal, but not isotropy-maximal.
Our main result, stated below in  Theorem \ref{main}, gives an equivariant diffeomorphism classification
of simply connected, non-negatively curved  manifolds admitting effective, isometric, almost isotropy-maximal torus actions, thus generalizing the main result in \cite{ES2}.

 We first introduce a crucial new tool for the proof of our main result:
we show that in our setting a strictly almost isotropy-maximal torus action can always be extended to a locally standard and isotropy-maximal action.

\begin{theoremlet}\label{ta}
\label{extend} 
Let $T^k$ act isometrically and effectively on $M^n$, a closed, simply connected, non-negatively curved Riemannian manifold. Assume that the action is strictly almost isotropy-maximal. Then the  $T^{k}$-action on $M^{n}$ may be uniquely extended to a smooth, effective $T^{k+1}$-action that is isotropy-maximal and  locally standard.
Moreover, $M^n/T^{k+1}$ and  all of its faces are diffeomorphic to a disk, after smoothing the corners. 
\end{theoremlet}

We also obtain the following two extensions of Theorem 1.1 of Wiemeler \cite{Wie}, and of Theorem B  of Galaz-Garc\'ia, Kerin, Radeschi, and Wiemeler \cite{GGKRW}. 

\begin{theoremlet}\label{RE} Let $M^n$ be a closed, rationally elliptic $n$-manifold admitting a smooth, effective, locally standard, and isotropy-maximal  $T^k$-action. 
Then 
$M^n$ is equivariantly homeomorphic to a quotient of a free linear torus action of 
\begin{equation}\label{E1}
\mathcal{Z}^m=\prod_{i<r} S^{2n_i} \times \prod_{i\geq r} S^{2n_i-1},\, \, n_i\geq 2, \,\,   \textrm{where} \,\,\,
 n\leq m\leq 3n-3k\,.
 \end{equation}
\end{theoremlet}

If we  add in the condition that the $4$-dimensional faces of the quotient are diffeomorphic to disks, after smoothing the corners, then we may improve the conclusion of Theorem \ref{RE} to one of  equivariant diffeomorphism.
\begin{theoremlet}\label{RE2} Let $M^n$ be a closed, rationally elliptic $n$-manifold admitting a smooth, effective, locally standard, and isotropy-maximal  $T^k$-action. Suppose all $4$-dimensional faces of $M^n/T^k$ are diffeomorphic to disks, after smoothing the corners.
Then 
$M^n$ is equivariantly diffeomorphic to a quotient of a free linear torus action of 
\begin{equation}\label{E1}
\mathcal{Z}^m=\prod_{i<r} S^{2n_i} \times \prod_{i\geq r} S^{2n_i-1},\, \, n_i\geq 2, \,\,   \textrm{where} \,\,\,
 n\leq m\leq 3n-3k\,.
 \end{equation}
\end{theoremlet}

We note that an isotropy-maximal action is also almost isotropy-maximal.
Theorems \ref{ta},  \ref{RE2}, and  the main result in \cite{ES2} (Theorem \ref{thma} in this paper) are instrumental in proving the following theorem.

\begin{theoremlet}\label{main}
 Let $T^k$ act isometrically and effectively on $M^n$, a closed, simply connected, non-negatively curved Riemannian manifold. Assume that the action is almost isotropy-maximal.  Then the following hold:
\begin{enumerate}
\item  $  \lfloor \frac{n+1}{2}\rfloor 
\leq k \leq \lfloor \frac{2n}{3}\rfloor$,  if the $T^k$-action is isotropy-maximal; 
\item  $  \lfloor \frac{n+2}{2} \rfloor -1
\leq k \leq \lfloor \frac{2n}{3}\rfloor - 1$,  if the $T^k$-action is strictly almost isotropy-maximal; and 
\item $M$ is equivariantly diffeomorphic to a quotient of a free linear torus action of a product of spheres as in Display \eqref{E1}, noting that when the action is strictly almost isotropy-maximal then  $n\leq m\leq 3n-3k-3$. 
\end{enumerate}
\end{theoremlet}

In the situation of Theorem \ref{main}, the torus action on the quotient $\mathcal{Z}^m/T'$, where $T'$ is a free linear torus action, is defined as follows. Let $T$ be a maximal torus of $\prod_{i<r} SO(2n_i+1) \times \prod_{i\geq r} SO(2n_i)$. Then there is a natural linear action of $T$ on $\mathcal{Z}^m$. Moreover, $T'$ can be identified with a subtorus of $T$ and $T/T'$ acts on
$\mathcal{Z}^m/T'\simeq M$.

We note that the upper and lower bounds in Parts $1$ and $2$ of Theorem \ref{main} follow by combining Theorem \ref{thma} and the corresponding lower bounds on the dimension, respectively, $2k-n \ge 0$ or $2k-n+1 \ge 0, $  of the orbit of smallest dimension. 
\begin{remark} When $k<\lfloor 2n/3\rfloor$, an almost isotropy-maximal $T^k$-action  
need not be isotropy-maximal.
Consider, for example, the action of $T^3$ on $S^3\times S^3$ given by 
$$((e^{i\theta_1}, e^{i\theta_2}, e^{i\theta_3}), (z_1, z_2, w_1, w_2)\mapsto (e^{i\theta_1}z_1, e^{i\theta_2}z_2, e^{i\theta_3}w_1, w_2), $$
where $S^3\times S^3=\{(z_1, z_2, w_1, w_2) \in \ccc^4: z_1\bar{z}_1 +z_2\bar{z}_2 =  w_1\bar{w}_1 +w_2\bar{w}_2=1\}$.
\end{remark}

Recall that the {\it symmetry rank} of a manifold is defined to be the rank of the isometry group of $M$. Closed manifolds of positive sectional curvature and maximal symmetry rank were classified in Grove and Searle  \cite{GS1}, where they found that  such manifolds are  diffeomorphic to 
spheres, real projective spaces, lens spaces, or complex projective spaces. In fact, they are  equivariantly diffeomorphic to these spaces with a linear torus action by work of Galaz-Garc\'ia \cite{GG2}.
For closed, simply connected, non-negatively curved manifolds, we only have the following conjecture (see \cite{ES2}, cf. Galaz-Garc\'ia and Searle \cite{GGS1}).

\begin{msrconj}\label{msrconj} Let $T^k$ act isometrically and effectively on
$M^n$, a closed, simply connected, non-negatively curved Riemannian manifold. Then  the following hold:
\begin{enumerate}
\item 
$k\leq \lfloor 2n/3\rfloor$; and 

\item When $k= \lfloor 2n/3\rfloor$, $M^n$ is equivariantly diffeomorphic to  
$\mathcal{Z}/T^m$ with a linear $T^k$-action,  where 
$$\mathcal{Z}=  \prod_{i\leq r} S^{2n_i-1} \times\prod_{i>r} S^{2n_i},$$
 with  $n_i\geq 2,  \,\,0 \leq m \leq 2n \mod 3,$
and the $T^m$-action on $\mathcal{Z}$ is  free and linear.
 \end{enumerate}
\end{msrconj}

\begin{remark}
The equivariant diffeomorphism classes of $M^n$ in Part 2 of the Maximal Symmetry Rank Conjecture can be described as follows:  

$$
\begin{array}{cll}
 M^{3m} &\simeq S^3 \times \cdots \times S^3, \\
 {}\\
M^{3m-1}& \simeq S^5 \times S^3 \times \cdots \times S^3 \quad \text{or} \quad 
                              M^{3m}/T^1, \\
{}\\                              
M^{3m-2} & \simeq \left\{\begin{array}{l}
                                  S^7 \times S^3 \times \cdots \times S^3 \quad \text{or}  \\
                                   S^5 \times S^5  \times S^3 \times \cdots \times S^3 \quad \text{or}  \\
                                  S^4  \times S^3 \times \cdots \times S^3 \quad \text{or}  \\
                                   M^{3m}/T^2 \quad \text{or} \quad M^{3m-1}/T^1\,,\\
                                 
                                   \end{array} \right.
                                                                    
 \end{array}$$
 \vskip .25cm
 \noindent
  where the $T^k$-action is linear.    In particular, the  
torus action on each of these manifolds is isotropy-maximal. 
 \end{remark}   

By Theorem B in \cite{ES2}, the Maximal Symmetry Rank Conjecture holds for isotropy-maximal actions. By Theorem 7.1 in  \cite{ES2}, for $k=\lfloor 2n/3\rfloor$, an almost isotropy-maximal $T^k$-action on a closed, simply connected, non-negatively curved manifold is, in fact, isotropy-maximal. So, Theorem \ref{main} sheds no new light on the Maximal Symmetry Rank conjecture.  However, it does represent a step forward in understanding the case of closed, simply connected, non-negatively curved manifolds of {\em almost maximal} symmetry rank, for which little is known. Indeed,  a classification for non-negatively curved manifolds of almost maximal symmetry rank is only known in dimensions less than or equal to 
$6$. In particular, for dimensions $2$ and $3$, the classification is well-known. In dimension $4$, the classification is due to Kleiner \cite{K}, Searle and Yang \cite{SY},  Grove and Searle \cite{GS1}, Galaz-Garc\'ia \cite{GG}, and Grove and Wilking \cite{GW}. In dimensions $5$ and $6$, the classification is due to, respectively, Galaz-Garc\'ia and Searle \cite{GGS2} and  Escher and Searle \cite{ES1}. For dimensions $5$ and $6$, Theorem \ref{main} gives us a simplified proof of the classification for those torus actions that are almost isotropy-maximal.

Finally, we note that in dimensions $4$ and $6$  all non-negatively curved manifolds of almost maximal symmetry rank already appear in the almost isotropy-maximal classification.  
However, in dimension $5$, the almost maximal symmetry rank classification also includes the Wu manifold, $SU(3)/SO(3)$, which does not admit an almost isotropy-maximal torus action. In fact, the Wu manifold is not even topologically a quotient by a free linear torus action on a product of spheres of dimension $3$ or larger:  the long exact sequence in homotopy associated to such a fibration gives us a contradiction, as $\pi_2(SU(3)/SO(3))\cong \zzz_2$.
In particular, it is not clear that this phenomenon only occurs in low dimensions or if it is an indication that in higher dimensions, a classification of closed, simply connected non-negatively curved manifolds  of almost maximal symmetry rank may include more than just those manifolds appearing in the almost isotropy-maximal classification. 

\subsection{Organization} The paper is organized as follows.  In Section \ref{s2}, we gather preliminary definitions and facts that are used throughout the paper. In Section \ref{s3}, we 
prove an analog of Proposition 6.1 in \cite{ES2} (Proposition \ref{lstd} here) for strictly almost isotropy-maximal actions.  In Section \ref{s4}, we show how to extend the isometric $T^k$-action to a smooth $T^{k+1}$-action, thus proving Theorem \ref{extend}. In Section \ref{s5}, we prove Theorems  
\ref{RE}, \ref{RE2}, and \ref{main}.
\begin{ack} A portion of this work draws from the PhD thesis of the first author. The authors would like to thank  Fred Wilhelm  for helpful comments and suggestions.  They also thank the referees for their careful reading of the paper and their suggestions for improvement.
C. Escher acknowledges partial support from the Simons Foundation (\#585481, C. Escher). C. Searle  acknowledges partial support from grants from the National Science Foundation (\#DMS-1611780, \#DMS-1906404, and \#DMS-2204324), as well as  from the Simons Foundation (\#355508, C. Searle). This material is based upon work supported by the National Security Agency
under Grant No. H98230-18-1-0144 while C. Escher and C. Searle were both in residence at the
Mathematical Sciences Research Institute in Berkeley, California, during the
summer of 2018. Finally, the authors would like to thank both the Fields Institute and the Max Planck Institute for Mathematics for their support during the final revisions of this paper.

\end{ack}

\section{Preliminaries}\label{s2}

In this section we will gather basic results and facts.  See \cite{ES2} for more details on many of the concepts in this section. 

\subsection{Transformation Groups}\label{2.1}

Let $G$ be a compact Lie group acting on a smooth manifold $M$. We denote by $G_x=\{\, g\in G : gx=x\, \}$ the \emph{isotropy group} at $x\in M$ and by $G(x)=\{\, gx : g\in G\, \}\simeq G/G_x$ the \emph{orbit} of $x$. 

We say that the $G$-action is \emph{effective} or \emph{almost effective} if $\bigcap_{x\in M} G_x$ is trivial or finite, respectively.
The action is said to be \emph{free} if every isotropy group is trivial and \emph{almost free} if every isotropy group is finite. 
The {\em free rank} of an action is the rank of the maximal subtorus that acts almost freely. 

 We will sometimes denote $M^G=\{\, x\in M : gx=x  \,\text{for all} \,g\in G \, \}$,  the \emph{fixed point set} of the $G$-action, by $\Fix(M ; G )$. Its dimension is defined as the maximum dimension of its connected components.  
 
 One measurement for the size of a transformation group $G\times M\rightarrow M$ is the dimension of its orbit space, $M/G$, also called the {\it cohomogeneity} of the action. This dimension is clearly constrained by the dimension of the fixed point set $M^G$  of $G$ in $M$. In fact, $\dim (M/G)\geq \dim(M^G) +1$ for any non-trivial, non-transitive action. In light of this, the {\it fixed-point cohomogeneity} of an action, denoted by $\textrm{cohomfix}(M;G)$, is defined by
\[
\textrm{cohomfix}(M; G) = \dim(M/G) - \dim(M^G) -1\geq 0.
\]
A manifold with fixed-point cohomogeneity $0$ is also called a $G$-{\it fixed-point-homogeneous manifold}.

Finally we  recall Theorem 9.1 of Chapter 1 of Bredon \cite{Br} which allows us to lift a group action to a covering space.

\begin{theorem}\cite{Br}\label{lift}
 Let $G$ be a connected Lie group acting effectively on a connected, locally path-connected space $X$ and let $X'$ be any covering space of $X$. Then there is a connected covering group $G'$ of $G$ with an effective action of $G'$ on $X'$ covering the given action. Moreover, $G'$ and its action on $X'$ are unique.
 
The kernel of $G'\rightarrow G$ is a subgroup of the group of deck transformations of $X'\rightarrow X$. In particular, if $X'\rightarrow X$
 has finitely many sheets, then so does $G'\rightarrow G$.
If $G$ has a stationary point in $X$, then $G' = G$ and $\Fix\,(X'; G)$ is the
full inverse image of $\Fix\,(X; G)$.
\end{theorem}

\subsection{Torus Actions} \label{s2.1}
 In this subsection we will recall notation and facts about smooth $G$-actions on smooth $n$-manifolds, $M$, in the special case when $G$ is a torus.
 We first recall the definitions of an {\em isotropy-maximal torus action} and an {\em almost isotropy-maximal torus action}.  

\begin{definition}[{\bf Isotropy-Maximal Action}]\label{IM} Let $M^n$ be a connected manifold with an effective $T^k$-action.  We call the $T^k$-action on $M^n$ {\em isotropy-maximal} provided  either of 
the following equivalent conditions hold:
\begin{enumerate}  
\item 
 There is a point $x \in M$ such that 
the dimension of its isotropy group is $n-k$, that is,  $\dim(T^k_x) = n-k$; or
\item  There is a point $x \in M$ such that $\dim(T^k(x)) = 2k-n$. 
\end{enumerate} 
\end{definition}

 Note that $n-k$ is the largest possible dimension of any isotropy subgroup of a $T^k$ action on $M$ and $2k-n$ is correspondingly the smallest possible dimension of any orbit.  
The following lemma in \cite{I} shows that the existence of an isotropy-maximal 
torus action on $M$ implies that there is no larger torus acting on $M$ effectively.

\begin{lemma}\label{ishida} \cite{I}  Let $M$ be a connected manifold with an effective $T^k$-action.  Let $T^l \subset T^k$ be a subtorus of $T^k$.  Suppose 
that the action of $T^k$ restricted to $T^l$ on $M$ is isotropy-maximal.  Then $T^l = T^k$.  
\end{lemma} We also define the concept of an almost isotropy-maximal action. 
\begin{definition}[{\bf (Strictly) Almost Isotropy-Maximal Action}]\label{aim}  Let $M^n$ be a connected manifold with an effective $T^k$-action.  
\begin{enumerate}  
\item We call the $T^k$-action on $M^n$ {\em almost isotropy-maximal} if 
\begin{enumerate}
\item There is a point $x \in M$ such that 
the dimension of its isotropy group is $n-k-1$, that is,  $\dim(T^k_x) = n-k-1$; or, equivalently
\item There is a point $x \in M$ such that
 $\dim(T^k(x)) = 2k-n+1$. 
\end{enumerate}
\item  We call a $T^k$-action {\em strictly} almost isotropy-maximal if the action is almost isotropy-maximal but not isotropy-maximal.
\end{enumerate} 
\end{definition}

The next lemma generalizes Lemma \ref{ishida} to almost isotropy-maximal actions.

\begin{lemma}\label{genIshida} Let $M$ be a connected manifold with an effective $T^k$-action.  Let $T^l \subset T^k$ be a subtorus of $T^k$.  Suppose 
that the action of $T^k$ restricted to $T^l$ on $M$ is almost isotropy-maximal.  Then $l=k$ or $l=k-1$.
\end{lemma}

 \begin{remark} Both Lemma \ref{ishida} and Lemma \ref{genIshida} hold for almost effective actions, since their respective proofs depend only on the toral rank of the action.
 \end{remark}
Finally, we recall the following result of Su \cite{Su} (cf. Hattori and Yoshida \cite{HY}), which we use in Case 2.a of Step 2 of the proof of Theorem \ref{ta} in order to lift a torus action on a simply-connected base space of a principal torus bundle to its total space.
\begin{theorem}\cite{Su}\label{Su}
Let $P\rightarrow X$ be a principal $T^l$-bundle  and suppose that $T^k$ acts on $X$. If $H^1(X; \zzz)=0$, then the $T^k$-action can be lifted to $P$.
\end{theorem}

\subsection{Torus Manifolds and Orbifolds} An important subclass of manifolds  admitting an effective, isotropy-maximal torus action are the so-called {\it torus manifolds}. 
For more details on torus manifolds, we refer the reader to Hattori and Masuda \cite{HM},  Masuda and Panov \cite{MP}, and Buchstaber and Panov \cite{BP}.

 \begin{definition}[{\bf Torus Manifold}] A {\em torus manifold} $M$ is a $2n$-dimensional closed, connected, orientable, smooth manifold with an effective smooth action of 
an $n$-dimensional torus $T$ such that $M^T\neq \emptyset$, or, equivalently, the $T$-action is isotropy-maximal.
\end{definition}

We recall the definition of a {\em characteristic submanifold} and generalize it.

\begin{definition}[{\bf (Generalized) Characteristic Submanifold}]
Let $T^k$ act smoothly and effectively on a closed manifold $M^n$ with $2k\geq n-1$. Let $F$ be a connected component of $\Fix(M; S^1)$ for some circle subgroup $S^1\subset T^k$. Then $F$ is called a {\em characteristic submanifold} of $M$ if $2k=n$ and it satisfies the following properties:

\begin{enumerate}
\item   $F$ is of codimension $2$  in $M$;  and
\item $F$ contains a $T^k$-fixed point. 
\end{enumerate}
More generally,
we say that $F$ is a {\em generalized characteristic submanifold} provided it satisfies Property $1$ and contains an orbit of dimension $2k-n$ or $2k-n+1$. 
\end{definition}

\begin{remark}\label{fphaim} 
Generalized characteristic submanifolds only occur in the presence of an almost isotropy-maximal torus action. Moreover, for any such action, if the rank of the isotropy subgroup of the smallest orbit 
of the action is equal to $m\geq 1$, there are $m$ generalized characteristic submanifolds containing that orbit, each corresponding to the fixed point set of some distinct circle subgroup of the isotropy subgroup.

We emphasize that the action of each such circle on $M$ is fixed-point-homogeneous. Moreover, 
given $x\in M$  such that $T(x)$ is an orbit of smallest dimension, for each generalized characteristic submanifold $F$ containing $T(x)$,   we observe that there is a chain of inclusions of totally geodesic submanifolds, that we may specify as follows. We let $$T(x)\subset F_{j}\subset F_{j-1}\subset \cdots\subset F_2\subset F_1=F\subset M$$ denote the chain of inclusions, where $j$ is equal to either $n-k$ or $n-k-1$, the subscript corresponds to the rank of the torus fixing the submanifold, and $\dim(F_i)=n-2i$, with $1\leq i\leq j$.
\end{remark}

An important class of $T^k$-actions on an $n$-dimensional manifold $M^{n}$ consists of {\em locally standard} torus actions, whose definition we now recall.

 \begin{definition}[{\bf Locally Standard}] \cite{ES2} \label{ls} 
A $T^k$-action on $M^n$ is called {\em locally standard} if for 
 each point $x \in M$, there is a neighborhood of $x$ in $M$ which is $T^k$-equivariantly diffeomorphic to
$$T^r \times W \times \rrr^m,$$
where $r=k-\dim(T^k_x)$, $W$ is a faithful $T^k_x$-representation of real dimension $2\dim(T^k_x)$,   and $T^k\cong T^r\times T^k_x$ acts trivially on $\rrr^m$, $T^r$ acts trivially on $W$, and $T^k_x$ acts trivially on $T^r$.
\end{definition}

We now recall the definitions of an orbifold and a torus orbifold.   For more details about orbifolds and actions of tori on orbifolds, see Haefliger and Salem \cite{HS1},  and \cite{GGKRW}.
\begin{definition}[{\bf Orbifold}] An {\em $n$-dimensional (smooth) orbifold}, denoted by $\mathcal{O}$, is a second-countable, Hausdorff topological space $|\mathcal{O}|$, called the underlying topological space of $\mathcal{O}$, together with an {\em equivalence class of $n$-dimensional orbifold atlases}.
\end{definition}
In analogy with a torus manifold, we may define a torus orbifold, as follows. 

\begin{definition}[{\bf Torus Orbifold}] A {\em torus orbifold}, $\mathcal{O}$, is a $2n$-dimensional, closed, orientable orbifold with an effective smooth action of 
an $n$-dimensional torus $T$ such that $\mathcal{O}^T\neq \emptyset$, or, equivalently, the $T$-action is isotropy-maximal.

\end{definition}

Recall that a closed, simply connected topological space is called {\em rationally elliptic} if $\pi_*(X)\otimes \qqq$ and $H_*(X; \qqq)$ are finite-dimensional $\qqq$-vector spaces.  In \cite{GGKRW}, they outline the proof of the following proposition in the proof of their Theorem A,  which we use in the proof of Theorem \ref{main}.

\begin{proposition}\cite{GGKRW}\label{REtorusorbifold} Let $\mathcal{O}$ be a closed, rationally elliptic, simply connected  $2n$-dimensional torus orbifold with a smooth torus action, $T^n$. Then the face poset of the quotient $\mathcal{O}/T$ is combinatorially equivalent to the face poset of 
$$P^n=\prod_{i<r} \Sigma^{n_i} \times \prod_{i\geq r} \Delta^{n_i},$$
where $\Sigma^{n_i} =S^{2n_i}/T^n_i$,
and $\Delta^{n_i}=S^{2n_i+1}/T^{n_i+1}$  is an $n_i$-simplex.
\end{proposition}

 The $T^{n_i}$-action on $S^{2n_i}$ is the suspension of the standard $T^{n_i}$-action on $\rrr^{2n_i}$, and     
$\Sigma^{n_i}$ is a  {\em lunar suspension} of $\Delta^{n_i-1}$, namely, it is obtained as the suspension of $\Delta^{n_i-1}$, ignoring the simplicial structure of $\Delta^{n_i-1}$.  Note that each $n_i$-simplex has $n_i+1$ facets and each $\Sigma^{n_i}$ has $n_i$ facets. The number of facets of $P^n$ in this case is bounded between $n$ 
and $2n$. 

\subsection{Geometric results in the presence of a lower curvature bound}

We now recall some general results about $G$-manifolds with positive and non-negative  sectional curvature which we will use throughout.
The first of these is  the Splitting Theorem of Cheeger and Gromoll.
\begin{theorem}\cite{CG}\label{splitting}
Let $M$ be a compact manifold of non-negative sectional curvature. Then $\pi_1(M)$ contains a finite normal subgroup $\psi$ such that $\pi_1(M)/\psi$ is a finite group extended by $\zzz^k$, and $\widetilde{M}$, the universal covering of $M$, splits isometrically as $\overline{M} \times \rrr^k$, where $\overline{M}$ is compact and non-negatively curved.
\end{theorem}

As mentioned in Remark \ref{fphaim}, a manifold admitting an almost isotropy-maximal torus action is an example of an $S^1$-fixed-point-homogeneous manifold. Closed, simply connected, fixed-point-homogeneous manifolds of positive curvature were classified in Grove and Searle \cite{GS}. More recently, in Spindeler \cite{Spi},  
the following characterization of closed, non-negatively curved fixed-point-homogeneous manifolds is given.

\begin{theorem}\label{Spindeler} \cite{Spi} Let $G$ be a compact Lie group. Assume that $G$ acts fixed-point-homogeneously  on a closed, non-negatively curved Riemannian manifold $M$. Let $F$ be a fixed point set component of maximal dimension. Then there exists a  smooth submanifold $N\subset M$, without boundary, 
such that $M$ is diffeomorphic to the normal disk bundles $D(F)$ and $D(N)$  
glued together along their common boundary, $E$, that is, 
\bdm
M  = D(F) \cup_{E} D(N).
\edm
Further, $N$ is $G$-invariant and all points of $M\, \setminus\, \{F \cup N\}$ belong to principal $G$-orbits.
\end{theorem}
 
\begin{remark}  In the above disk bundle decomposition, the projection maps $\pi_F: E\rightarrow F$ and $\pi_N: E\rightarrow N$ are both $G$-equivariant.
 \end{remark}

Lemma 3.29 in \cite{Spi}, which we include here, shows that in the special case where $M$ is simply connected, the dimension of $N$ is bounded above. Although it was originally stated for manifolds of non-negative sectional curvature with an isometric $G$-action, the proof of the lemma shows that it also holds in the smooth setting. 
\begin{lemma}\cite{Spi}\label{Spindelerlemma} Let $M$ be a smooth $G$-fixed-point-homogeneous manifold which decomposes as in Theorem \ref{Spindeler}. Assume that $M$ is simply
connected and $G$ is connected. Then 
the codimension of $N$ in $M$ is greater than or equal to two.
\end{lemma}

The following theorem from \cite{ES1} gives us topological information about the fundamental groups of $E$, $F$, and $N$.

 \begin{theorem}\label{ES1}\cite{ES1} Let $M^n$ be a simply connected manifold that decomposes as the union of two disk bundles as follows:
$$M^n=D^{k_1}(N_1)\cup_E D^{k_2}(N_2).$$
If $k_1=k_2=2$, then $\pi_1(N_1)$ and $\pi_1(N_2)$ are cyclic groups.

Moreover, 
\begin{enumerate}
\item If $k_i=2$, $\pi_2(N_i)=0$,  for $i=1, 2$ and $\pi_ 1(N_i)$  is infinite for some $i\in\{1, 2\}$, then 
$\pi_1(E)\cong \zzz^2$.

\item If $k_i\geq 3$, for some $i\in \{1, 2\}$,  then $\pi_1(E) \cong \pi_1(N_i)$ and $\pi_1(N_{i+1}) = 0$, with  the indices  taken mod $2$.
\end{enumerate}
\end{theorem}

 We now claim that for $k_i\geq 2$ for some $i\in \{1, 2\}$, the fundamental groups of each of the $N_i$ are cyclic under the same hypotheses as in Theorem \ref{ES1}. We see this as follows. Using the arguments  in the proof of Lemma 3.29 in \cite{Spi} (Lemma \ref{Spindelerlemma} here), we see that if $k_i\geq 2$ for some $i\in \{1, 2\}$, then $k_{i+1}\geq 2$. Part 1 of Theorem \ref{ES1} gives us the result for $k_i=k_{i+1}=2$. If $k_i\geq 3$ for some $i\in \{1, 2\}$, using the results of Part 2 of Theorem \ref{ES1} and the long exact sequences in homotopy  associated to $S^{k_1-1}\rightarrow E\rightarrow N_1$ and $S^{k_2-1}\rightarrow E\rightarrow N_2$, the result follows.  Hence we obtain the following corollary.
 
\begin{corollary}\label{Ncyclic}  
Let $M^n$ be a simply connected manifold that decomposes as the union of two disk bundles as follows:
$$M^n=D^{k_1}(N_1)\cup_E D^{k_2}(N_2),$$
with $k_i\geq 2$ for some $i\in \{1, 2\}$.
Then the fundamental groups of $N_1$ and $N_2$ are cyclic.
\end{corollary}

Combining Proposition 4.5 from \cite{Wie}, Proposition \ref{REtorusorbifold}, and Theorem 4.2 in \cite{D}, 
we obtain the following result that characterizes the orbit space of a simply connected torus orbifold  that additionally is either rationally elliptic or non-negatively curved. Recall that a nice manifold with corners is a manifold with corners such that every codimension $k$ face is contained in exactly $k$ facets, see Davis \cite{D}.  

\begin{proposition}\label{P}\cite{D, GGKRW, Wie}
Let $X$ be a closed, simply connected torus orbifold that is either rationally elliptic or non-negatively curved. Then the quotient space, 
 $X^{2n}/T^n$,  is a nice manifold with corners all of whose faces 
are acyclic,  and  $X^{2n}/T^n$ is combinatorially equivalent to
\begin{equation}\label{P^n}
P^n=\prod_{i<r} \Sigma^{n_i} \times \prod_{i\geq r} \Delta^{n_i}.
\end{equation}
Further, if we assume that all four-dimensional faces of $X^{2n}/T^n$ are diffeomorphic to disks, after smoothing the corners, then $X^{2n}/T^n$ is diffeomorphic to $P^n$. 
 \end{proposition}

For $M^n$, a closed, simply connected, non-negatively curved $n$-manifold admitting an isometric, effective, and isotropy-maximal $T^k$-action, the following Proposition 6.1 from \cite{ES2} 
gives us information about the structure of the quotient space, $M^n/T^k$, as well as a complete description of the corresponding isotropy groups.
\begin{proposition}\label{lstd}\cite{ES2} Let $T^k$ act isometrically, effectively, and isotropy-maximally on $M^{n}$, a closed,  simply connected, non-negatively curved Riemannian manifold. Then the following are true:
\begin{enumerate}
\item The torus action on $M$ is locally standard,
in particular,  $M/T$ is a nice manifold with corners, such that the isotropy groups are constant on all open faces of $M/T$; and
\item $M/T$ and all closed faces of $M/T$ are diffeomorphic to standard disks, after smoothing the corners.
\end{enumerate}
\end{proposition}

We also make use of the following theorem, which follows by combining Theorem D  from \cite{GGKRW} with Corollary 2.37 from \cite{ES2}. 
\begin{theorem}\cite{ES2, GGKRW}\label{misre} Let $M$ be a closed, simply connected, non-negatively curved Riemannian manifold admitting an effective, isometric, almost isotropy-maximal torus action. Then $M$ is rationally elliptic.
\end{theorem}

Finally, we recall Theorem A in \cite{ES2}, which we need for the proof of Theorem \ref{main}.

\begin{theorem}\label{thma}\cite{ES2}  Let $T^k$ act isometrically and effectively on $M^n$, a closed, simply connected, non-negatively curved Riemannian manifold. Assume that the action is isotropy-maximal.  Then the following hold:
\begin{enumerate}
\item $  \lfloor(n+1)/2\rfloor 
\leq k\leq \lfloor 2n/3\rfloor$; and 

\item $M$ is equivariantly diffeomorphic to a quotient by a free linear torus action of 
$$\mathcal{Z}^m=\prod_{i<r} S^{2n_i} \times \prod_{i\geq r} S^{2n_i-1},\, \, n_i\geq 2, \,\,   \textrm{where} \,\,\,
 n\leq m\leq 3n-3k\,.$$
\end{enumerate}
\end{theorem} 
\noindent The torus action referred to in Part 2 of Theorem \ref{thma} is defined as in Theorem \ref{main}.  

\begin{remark}\label{Wie}
It has come to our attention that there is a gap in the proof of Theorem 3.7 in \cite{ES2}. A result filling that gap can be found in Theorem 1.1 in Wiemeler \cite{Wie2}. Since Theorem 3.7 is used in the proof of Theorem A in \cite{ES2} to improve the equivariant homeomorphism classification to an equivariant diffeomorphism classification, it is important to clarify that with Theorem 1.1 in \cite{Wie2}, the equivariant diffeomorphism classification results in Theorem A in \cite{ES2} still hold.  
\end{remark}

\section{Almost Isotropy-Maximal Torus Actions}\label{s3}

In this section, our goal is to prove a partial  analog of Proposition \ref{lstd} for strictly almost isotropy-maximal actions.   
We first establish some notation. Let $\mathcal{M}_0^{T^k}(n)$ denote the class of  closed,  non-negatively curved Riemannian $n$-manifolds admitting an isometric and effective $T^k$-action.    For simplicity of notation, we  let $T$  denote the torus $T^k$, when the rank 
is clear from context.
 If the $T^k$-action on $M$ is also almost isotropy-maximal, by Remark \ref{fphaim}, the torus action is $S^1$-fixed-point-homogeneous with generalized characteristic submanifold 
$F\subset M$ fixed by a circle subgroup $C\subset T^k$, containing 
a smallest orbit. 
 In particular,  the action of $T/C$ on $F$ is also almost isotropy-maximal.
By Theorem \ref{Spindeler}, $M$ admits  an equivariant disk bundle decomposition as 
 \begin{equation} \label{disk} 
 M =D(F)\cup_E D(N),
 \end{equation}
where $N$ is an invariant submanifold at maximal distance from $F$. By Lemma \ref{Spindelerlemma},  $\codim(N) \geq 2$,
and by Corollary \ref{Ncyclic},  the fundamental groups of $F$ and $N$ are cyclic.

Observe that when the action is not isotropy-maximal, there are cases where the  
torus action on $M$ 
may not be locally standard everywhere, due to the existence of fixed point sets of disconnected groups.  Hence,
 we introduce the following general property.

\begin{CondC}\label{condC} 
Let $M\in \mathcal{M}_0^{T^k}(n)$ and let $M$ be $S^1$-fixed-point-homogeneous. Let $F$ and $N$ be as in Display \ref{disk}.  If  $N$ is either the fixed point set component of  a connected subgroup of $T^k$ or fixed by no subgroup of $T^k$, we say {\em $N$ satisfies Property C}. 
\end{CondC}

Recall that by Remark \ref{fphaim}, for a strictly almost isotropy-maximal torus action on a closed manifold, the chain of inclusions of totally geodesic submanifolds is given as follows:   $$T(x)\cong T^{2k-n+1}\subset F_{n-k-1}\subset F_{n-k-2}\subset \cdots\subset F_2\subset F_1=F\subset M,$$ where the subscript corresponds to the rank of the torus fixing the submanifold and $\dim(F_i)=n-2i$.

The proof of the following lemma is straightforward and left to the reader.
\begin{lemma}\label{F_i} Let $M\in \mathcal{M}_0^{T^k}(n)$ and suppose that the $T^k$-action is strictly almost isotropy-maximal. Let $F$ denote a generalized characteristic submanifold in $M$. Then  the induced torus action on each $F_i$ in the chain of inclusions of totally geodesic submanifolds is strictly almost isotropy-maximal and of cohomogeneity strictly less than $n-k$. In particular, it follows that each $F_i$ decomposes as
\begin{equation}\label{diskfi}
F_i=D(F_{i-1})\cup D(N_{F_i}).
\end{equation}
\end{lemma}

Before we state Theorem \ref{6.1analog}, we need to make one more definition. Observe that $N\subset M$ and each $N_{F_i}\subset F_i$ satisfy  \hyperref[condC]{Property C}
for an isometric, effective, isotropy-maximal torus action on a closed, simply-connected, non-negatively curved manifold.  
 However, for {\em strictly} almost isotropy-maximal torus actions, this need not be the case. To avoid this situation when working with strictly almost isotropy-maximal torus actions, we next define the notion of a {\em Property C manifold}. 

\begin{definition}[{\bf Property C manifold}]Let  $M\in \mathcal{M}_0^{T^k}(n)$  and suppose that the $T^k$-action is strictly almost isotropy-maximal.   For any generalized characteristic submanifold $F$ of $M$, suppose that in the corresponding disk bundle decompositions as in Displays \ref{disk} and \ref{diskfi},  $N$ and $N_{F_{i}}$, $1\leq i\leq n-k-1$,  satisfy \hyperref[condC]{Property C}. We then say that $(M, T^k)$ is a {\em Property C manifold}.
\end{definition}

We are now in a position to state Theorem \ref{6.1analog}. Before we do, we remark that  all disks 
in this section are standard, that is, they have the standard smooth structure.

\begin{theorem}\label{6.1analog} Let $(M, T^k)$ be a Property C manifold with $M$ simply connected. Then the following hold
\begin{enumerate}
\item The action of $T^k$ on $M^n$ is locally standard; and
\item $M/T$ is diffeomorphic  to a  disk, after smoothing the corners. 

\end{enumerate}
\end{theorem}

\begin{remark}
If $(M, T^k)$ is not a Property C manifold then the torus action on $M$ can fail to be locally standard in numerous ways, both on $D(N)$ and on $F$, and hence, in the latter case, on $D(F)$.
\end{remark}

The proof of Theorem \ref{6.1analog} follows along the same lines as the proofs of the main theorem in Dong \cite{Do} and Proposition  \ref{lstd}, and  is 
 by induction on the dimension of the orbit space.    
 We establish it in four steps. The first step is to prove Proposition \ref{cohom2}, the anchor of the induction. The second and third steps are  to prove Theorem \ref{p1} and Theorem \ref{p2}, which establish an analogue of Theorem \ref{6.1analog} for $F$ and for $N$, respectively.
  The fourth and last step is to prove Parts 1 and 2 of Theorem \ref{6.1analog}. 

 \subsection{Step 1 of the Proof of Theorem \ref{6.1analog}}

 Since a strictly almost isotropy-maximal cohomogeneity one torus action only has isotropy subgroups of rank $0$, the $T$-action on $M$ must be almost free and so $M/T$ is a circle. But, by Corollary 6.3 of Chapter 2 in \cite{Br}, this contradicts the hypothesis that $M$ is simply connected  and hence this case does not occur. Therefore, 
 the base case for the induction is when the action is by cohomogeneity two.  We establish this case in the following proposition.
\begin{proposition}\label{cohom2} Theorem \ref{6.1analog} holds for the case $n-k=2$.
\end{proposition} 
\begin{proof}
 As the action is strictly almost isotropy-maximal, the largest isotropy subgroup has rank $1$. 
 By Theorem 8.6 of Chapter 4 in \cite{Br}, the quotient space is a closed $2$-disk, $D^2$, whose interior corresponds to principal orbits and whose boundary, $\partial D^2$, corresponds to the singular orbits.  In particular, there is only one singular isotropy group, $T^1$, and $\partial D^2$ is the image of its fixed point set in the quotient space.
 As we saw above, $M$ decomposes as a union of disk bundles over the  generalized characteristic submanifold, $F$,  and $N$, the submanifold at maximal distance from $F$.  Since the image of $N$ in $M/T^k$ corresponds to an interior point of $D^2$, $N$ is a principal orbit and so $N=T^{k}$.   
 By Theorem \ref{ES1}, $\pi_1(N)$ is cyclic, so 
 $k=1$ and so $n=3$. 
 That is, we 
need only prove Theorem \ref{6.1analog} for an isometric, strictly almost isotropy-maximal $T^1$-action on a closed, simply connected, non-negatively curved $3$-manifold.

The action of the circle is $S^1$-fixed-point-homogeneous, and using work of  \cite{GG} we see  that $M^3$ is equivariantly diffeomorphic to $S^3$ with a linear $T^1$-action and decomposes as a union of disk bundles over $F$ and $N$, both of which are circles, with $F$ being fixed by the $T^1$-action and $N$ being a $T^1$-orbit.  In particular,   $F/T$ is diffeomorphic to $S^1$ and $N/T$ is a point. 

To show that the $T^1$-action on $S^3$ is locally standard, we note that since $F$ is a fixed circle in $M^3$, the Slice theorem gives us directly that the action in a neighborhood of $F$ is locally standard. Since $N$ is a principal orbit, its isotropy group is trivial. 
 It is then clear that the action is 
 locally standard in a  neighborhood of $N$.   Since $E$ is a principal circle bundle over $F$, the $T^1$-action on $E$ is also locally standard.   
Via the disk bundle decomposition, we see that the $T^1$-action is locally standard on $M$, which finishes the proof.
\end{proof}

\subsection{Step 2 of the Proof of Theorem \ref{6.1analog}}

In this subsection, our goal is to prove the following theorem, which establishes an analogue of Theorem \ref{6.1analog} for $F$.

\vspace{2mm}
\begin{theorem}\label{p1} Let $(M, T^k)$ be a Property C manifold with $M$ simply connected.  Assume that Theorem \ref{6.1analog} holds for torus actions of cohomogeneity $<n-k$.  Then  the induced torus action on $F$ is locally standard and $F/T$ is diffeomorphic to a disk or a product of a disk with a circle after smoothing the corners.
\end{theorem}

\vspace{2mm}

Before we begin the proof of Theorem \ref{p1}, we note that neither $F$ nor $N$ need be simply connected, but by Corollary \ref{Ncyclic}, they both have cyclic fundamental group. We prove the following technical lemma and proposition, which allow us to treat the cases where their respective fundamental groups are cyclic.
\begin{lemma}\label{Fpi1nontrivial} Let  $M\in \mathcal{M}_0^{T^k}(n)$  and suppose that the $T^k$-action is strictly almost isotropy-maximal.  Assume that $\pi_1(M)$ is cyclic and denote by $\widetilde{M}$
the Riemannian universal cover of $M$.  Then one of the following holds:
\begin{enumerate} 
\item If $\pi_1(M)\cong\zzz_q$, $\widetilde{M}$ is a closed, simply connected, non-negatively curved manifold admitting an isometric, strictly almost isotropy-maximal $T^k$-action;  and
\item If $\pi_1(M)\cong\zzz$, $\widetilde{M}=\overline{M}\times \rrr$, and $\overline{M}$ is 
a closed, simply connected non-negatively curved manifold admitting an isometric, isotropy-maximal $T^k$-action or  a strictly almost isotropy-maximal $T^{k-1}$-action.
\end{enumerate}
\end{lemma}

\begin{proof}

We begin with the proof of Part 1. By Theorem \ref{splitting}, $\widetilde{M}$, the universal cover of $M$,    is a closed, simply connected, non-negatively curved $n$-dimensional manifold. Let $\pi: \widetilde{M}\rightarrow M$ be the covering map.
 By Theorem \ref{lift}, it follows that the $T^k$-action on $M$ lifts to a  $T^k$-action on $\widetilde{M}$, since the lifted group   finitely covers  $T^k$.  Since the $T^k$-action on $M$ is strictly almost isotropy-maximal, there is an $x\in M$ with $T^{n-k-1}$ isotropy. Restricting to 
 this subaction, we may apply Theorem \ref{lift} to see that  $\pi^{-1}(x)$ is also fixed by $T^{n-k-1}$. Hence the $T^k$-action on $\widetilde{M}$ is strictly almost isotropy-maximal, as desired.

We now prove Part 2. By Theorem \ref{splitting}, the universal cover of $M$ splits isometrically as $\widetilde{M}=\overline{M}\times \rrr$, where $\overline{M}$ is a closed, simply connected, $(n-1)$-dimensional manifold of non-negative curvature. By Theorem \ref{lift}, since the kernel of the covering map $\pi: \widetilde{T}^k\rightarrow T^k$ is a subgroup of the group of deck transformations of $\pi: \widetilde{M}\rightarrow M$, $\widetilde{T}^k$ will be of the form $T^l \times \rrr^s$ where $l+s=k$ and $s\in \{0, 1\}$.  

Theorem $1$ of Hano \cite{Ha} states that  the connected component of the identity of the isometry group of a simply connected Riemannian product manifold is the product of the  identity components  of the isometry groups of the corresponding factors. Thus, we have $$\Isom_0(\widetilde{M})\cong \Isom_0(\rrr)\times \Isom_0(\overline{M}), $$ and so
$$\Isom_0(\widetilde{M})\cong \rrr\times \Isom_0(\overline{M}).$$ If $\widetilde{T}^k=T^k$, then $T^k\subset \Isom_0(\overline{M})$ and any non-trivial orbit of the $T^k$-action on $\widetilde{M}$ must lie entirely in the $\overline{M}$ factor.
 That is, the $\widetilde{T}^k$-action  leaves $\overline{M}$ invariant and fixes the $\rrr$-factor. 
As in the proof of Part 1, the $\widetilde{T}^k$-action on $\widetilde{M}$ is strictly almost isotropy-maximal, and thus the 
 $T^k$-action is isotropy-maximal on $\overline{M}$. 
Now assume that $\widetilde{T}^k=T^{k-1}\times \rrr$.  Again, as the $T^k$-action on $M$ is strictly almost isotropy-maximal,  
it follows that the $T^{k-1}$-action on $\overline{M}^{n-1}$ is strictly almost isotropy-maximal.
 This concludes the proof of Part 2.
\end{proof}

\smallskip

\begin{proposition}\label{Ftilde} Let $M\in \mathcal{M}_0^{T^k}(n)$ with a strictly almost isotropy-maximal $T^k$-action. Assume that Theorem \ref{6.1analog} holds for torus actions of cohomogeneity $< n-k$.
Let $A^{n-2l}$ be a component of the fixed point set of some $T^l\subset T^k$.  Suppose that $(A, T^{k-l})$ is a 
Property C (sub)manifold and $\pi_1(A)$ is cyclic and non-trivial.
Then, the $T^{k-l}$-action on  $A$ is locally standard and one of the following holds:
\begin{enumerate}
\item If $\pi_1(A)$ is finite, then  $A/T^{k-l}$ is diffeomorphic  to $D^{n-k-l}$ after smoothing the corners; and 
\item If $\pi_1(A)\cong \zzz$, then  $A/T^{k-l}$ is diffeomorphic  to $S^1\times D^{n-k-l-1}$ or to $D^{n-k-l}$ after smoothing the corners.
\end{enumerate} 
\end{proposition}

\begin{proof}
We first show that the  induced $T^{k-l}$-action on  $A^{n-2l}$ is locally standard.
Let $\widetilde{A}$ be the universal cover of $A$. 
Then we see from the proof of Lemma  \ref{Fpi1nontrivial} that $\widetilde{A}$ admits a strictly almost isotropy-maximal action.

If $\pi_1(A)\cong\zzz_q$, then the $T^{k-l}$-action on $A$ lifts to a $T^{k-l}$-action on  $\widetilde{A}$ and the covering map $\pi: \widetilde{A}\rightarrow A$ is equivariant with respect to these actions.  Thus  $(\widetilde{A}, T^{k-l})$ will also be a Property C manifold and the hypothesis gives us that the $T^{k-l}$-action on $\widetilde{A}$ is locally standard. Since being locally standard is a local condition, it immediately follows  that the induced action of $T^{k-l}$ on $A$ is locally standard. 

If $\pi_1(A)\cong\zzz$, then by Theorem \ref{splitting}, $\widetilde{A}=\bar{A}\times \mathbb{R}$. If the lifted group is $T^{k-l}$, then the action is isotropy-maximal on  $\bar{A}$. Combining Proposition \ref{lstd} 
with the fact that the torus action on the $\rrr$-factor is trivial yields that the $T^{k-l}$-action on $\widetilde{A}$ is locally standard and again, since being locally standard is a local condition, it immediately follows  that the induced action of $T^{k-l}$ on $A$ is locally standard.

If instead  the lifted group is $T^{k-l-1}\times \rrr$, then the action of $T^{k-l-1}$ on $\bar A$ is strictly almost isotropy-maximal. Moreover, we see, as in the case when the fundamental group is finite, that $(\bar A, T^{k-l-1})$ is a Property C manifold. Thus the action is locally standard on $\bar{A}$ by hypothesis. Since $T^{k-l-1}$  acts trivially on the $\rrr$-factor of $\widetilde{A}$, it follows that the   $T^{k-l-1}$-action on $\widetilde{A}$ is then locally standard.  To see that the $T^{k-l}$-action on $A$ is locally standard, it remains to show that the commuting $T^1$-action that lifts to an $\rrr$-action on the $\rrr$-factor of $\widetilde{A}$ is free.  But this follows immediately from Theorem \ref{lift}, for if the corresponding $T^1$-action on $A$ has non-trivial isotropy, then the lift of that isotropy subgroup 
would be compact, contradicting the fact that $T^1$ lifts to $\rrr$.

We now prove Part 1.   
By Part 1 of Lemma \ref{Fpi1nontrivial}, $\widetilde{A}$  is a closed, simply connected, non-negatively curved   manifold admitting an isometric,  strictly almost isotropy-maximal torus action and so by hypothesis  $\widetilde{A}/T^{k-l}$ is diffeomorphic to a disk. 
By letting $p:\widetilde{A}\rightarrow A$, $p_{\scaleto{A}{4pt}}:A\rightarrow A/T^{k-l}$, and $p_{\scaleto{\widetilde A}{5pt}}:\widetilde{A}\rightarrow \widetilde{A}/T^{k-l}$ it follows that the induced map 
$$p_{\scaleto{A}{4pt}}\circ p\circ p_{\scaleto{\widetilde A}{5pt}}^{-1}:\widetilde{A}/T^{k-l}\rightarrow A/T^{k-l}$$ is a surjective local diffeomorphism between compact Hausdorff spaces, and hence it is a covering map. 
Since  $\widetilde{A}/T^{k-l}$ is diffeomorphic to a disk, it must be a regular covering map.   If the covering is $m$-sheeted, then there exists a non-trivial element in the deck group acting on $m$ points of the fiber.  By Brouwer's fixed point theorem any such transformation should have a fixed point in the disk.  As all deck transformations act freely, this transformation must be trivial.  Hence $m=1$ and $p_{\scaleto{A}{4pt}}\circ p\circ p_{\scaleto{\widetilde A}{5pt}}^{-1}$ is a diffeomorphism after smoothing the corners. The result then follows. 

It remains to prove  Part 2.
In the case where the lifted group is $T^{k-l}$ on $\widetilde{A}$,  by Part 2 of Lemma \ref{Fpi1nontrivial} and its proof,  $T^{k-l}$ acts isotropy-maximally on $\bar{A}$ with $T^{k-l}$ acting trivially on the $\rrr$-factor.  
By hypothesis, $\bar{A}/T^{k-l}$ is diffeomorphic to a disk. 
Thus  $\widetilde{A}/T^{k-l} \simeq  D^{n-k-l}\times \rrr$. 
Since $p_{\scaleto{A}{4pt}}\circ p\circ p_{\scaleto{\widetilde A}{5pt}}^{-1}:\widetilde{A}/T^{k-l}\rightarrow A/T^{k-l}$  is a local diffeomorphism and $\widetilde{A}/T^{k-l}$ is simply connected, the path-lifting property holds. Hence, 
 $p_{\scaleto{A}{4pt}}\circ p\circ p_{\scaleto{\widetilde A}{5pt}}^{-1}$ is a covering map (see, for example,  the  proof of Proposition 6 of Section 5.6 in DoCarmo \cite{DoC}).
 Since $A$ is closed, it follows that $A/T^{k-l}$ is diffeomorphic   to $D^{n-k-l-1}\times S^1$, after smoothing the corners.

If instead the lifted group is $T^{k-l-1}\times \rrr$, then $(\bar{A}\times \rrr)/(T^{k-l-1}\times \rrr)=\bar{A}/T^{k-l-1}$. Since the $T^{k-l-1}$-action on $\bar A$ is strictly almost isotropy-maximal,  the  hypothesis implies that $\bar{A}/T^{k-l-1}$ is diffeomorphic to a disk, after smoothing the corners. It remains to show that $\widetilde{A}/(T^{k-l-1}\times \rrr)$ is diffeomorphic to $A/T^{k-l}$, but this follows as in the proof of Part 1. 
\end{proof}

\begin{proof}[Proof of Theorem \ref{p1}]
Since $F$ is a closed, totally geodesic submanifold of $M$, it is non-negatively curved. Moreover, $F$ admits a strictly almost isotropy-maximal $T^{k-1}$-action  of cohomogeneity strictly less than $n-k$. Since $(M, T^k)$ is a Property C manifold, by definition, so is $(F, T^{k-1})$. 
If $F$ is simply connected, then  by hypothesis, 
 the $T^{k-1}$-action on $F$  is locally standard and  $F/T^{k-1}$ is diffeomorphic to a standard disk $D^{n-k-1}$ after smoothing the corners. 
It remains to consider the case where $F$ is not simply connected. Since $\pi_1(F)$ is cyclic, the result follows from  
Proposition \ref{Ftilde}.  
\end{proof}

\subsection{Step 3 of the Proof of Theorem \ref{6.1analog}}
Our goal in this subsection is to prove the following theorem, which establishes an analogue of Theorem \ref{6.1analog} for $N$.

\begin{theorem}\label{p2} Let $(M, T^k)$ be a Property C manifold with $M$ simply connected.  Assume that Theorem \ref{6.1analog} holds for torus actions of cohomogeneity $<n-k$.
Then the induced torus action on $N$ is locally standard and $N/T$ is diffeomorphic to a disk or the product of a disk and a circle, after smoothing the corners.
\end{theorem}

 Before we begin the proof of Theorem \ref{p2}, we generalize Proposition \ref{lstd} to closed, non-negatively curved manifolds with cyclic fundamental group
admitting an isotropy-maximal torus action in Proposition \ref{lstdcyclic} below. We leave the details of the proof to the reader, observing that the proof of Part 1 of Proposition \ref{lstdcyclic} follows along the lines of the proof of Case 2 of Proposition \ref{lstd}.
To obtain the proof of Part 2 of Proposition \ref{lstdcyclic}, $M/T$ can be shown to be diffeomorphic to a disk by 
 applying the corresponding arguments of the  proof of Proposition \ref{Ftilde}, using the fact that when $\pi_1(M)\cong \zzz$ the lifted group can only be $T^{k-1}\times \rrr$, with the $\rrr$-factor acting freely on the $\rrr$-factor in $\widetilde{M}$. Since a diffeomorphism of nice manifolds with corners preserves faces, we see that  all closed faces of $M/T$ are diffeomorphic to disks. With this, we obtain the following proposition.

\begin{proposition}\label{lstdcyclic} Let $M\in \mathcal{M}_0^{T^k}(n)$ with an isotropy-maximal $T^k$-action. Suppose that $M$ has cyclic fundamental group. 
Then the following are true:
\begin{enumerate}
\item The torus action on $M$ is locally standard,
in particular,  $M/T$ is a nice manifold with corners, such that the isotropy groups are constant on all open faces of $M/T$; and
\item $M/T$ and all closed faces of $M/T$ are diffeomorphic to  
disks, after smoothing the corners.
\end{enumerate}
\end{proposition}

The proof of Theorem  \ref{p2} breaks into two cases.
Observe that $\pi_{F}: E \rightarrow F$ is a $C$-bundle over $F$. Let $x\in M$ be chosen so that $T(x)\cong T^k/T^{n-k-1}\cong T^{2k-n+1}$ is an orbit of smallest dimension in $M^n$, and $x'\in E$ so that $\pi_{F}(x')=x$. The orbit $T(x')$ in $E$  is an orbit of type $T/T'$ where $T'$ is some codimension-one subtorus of $T_x\cong T^{n-k-1}$.
Since the $T/C$-action on $F$ is locally standard by Theorem \ref{p1} and $T(x)$ is an orbit of smallest dimension, there is a neighborhood of $T(x')$ in $M$ which is equivariantly diffeomorphic to \begin{equation} \label{e1}(T/T' \times W \times \rrr)\times \rrr=T/T' \times W \times\rrr^2, \end{equation}
where $W$ is a faithful $T'$-representation of dimension $2 \dim(T')$ and $T$ acts trivially on $\rrr^2$ since one of the $\rrr$-factors is normal to $E$.
Let $T(y)$ be the projection of the orbit $T(x')$ to $N$. Since $\dim(T(x')) = \dim(T(x)) + 1$, there are two cases:
\begin{enumerate}
\item[] {\bf Case ${\bf 1}$}\label{case1}:
$T(y)$ is an orbit of smallest dimension, that is, $\dim(T(y)) = \dim(T(x))$; or
 \item[] {\bf Case ${\bf 2}$}\label{case2}: $T(y)$ is not an orbit of smallest dimension, that is, $\dim(T(y))=\dim(T(x))+1$.
\end{enumerate}

We further subdivide Case $2$ into two subcases, as follows.
Since $\pi_N$ is an equivariant map, it follows that $T(y)$ is a $T$-orbit of type $T/(H_0\times T')$, where $H_0\subset C$, is a finite subgroup. By a similar argument as above, we see that $T(y)$ has an invariant neighborhood in $M$ equivariantly diffeomorphic to 
\begin{equation}
\label{e2}
T/T'\times_{H_0} W \times\rrr^2\,,
\end{equation}
where $T'$ acts effectively on $W$ and the $H_0$-action on $W \times \rrr^2$ commutes with the $T'$-action. 
Moreover, since $\pi_N$ is an equivariant submersion, there is an  $\rrr$-factor  normal to $N$ because there was an $\rrr$-factor normal to $E$ in Display \ref{e1}.

 The normal bundle of the orbit $T(y)$ in $N$ is an invariant sub-bundle of the normal bundle of the orbit $T(y)\subset M$.  By the Slice theorem, the normal bundle of the tangent bundle to the orbit $T(y)$  is 
 isomorphic to the homogeneous vector bundle $T/T'\times_{H_0} W \times \rrr^2$.

Since $T'\cong T^{n-k-2}$  acts effectively and isotropy-maximally on $W$, the invariant sub-bundles of this normal bundle are of the form
\begin{equation}
\label{e3}
T/T'\times_{H_0} W' \times\rrr^j,
\end{equation}
where $W'$ is a faithful $T''$-representation of dimension $2 \dim(T'')$, with $T''\subset T'$, and $0\leq j\leq 2$.  Since one  $\rrr$-factor is normal to $N$, the dimension of $N$ must be either   $2k-n+2+2\,\dim(T'')+1$ or $2k-n+2+ 2\,\dim(T'')$. In the first case, the codimension of $N$ is $2\,\dim(T')-2\,\dim(T'')+1$ and in the second it is $2\,\dim(T')-2\,\dim(T'')+2$.
 
 The two subcases are then:
 \begin{enumerate}
\item[] {\bf Case ${\bf 2.a}$:} The codimension of $N$ is odd and greater than or equal to three; and
\item[] {\bf Case ${\bf 2.b}$:}  The codimension of $N$ is even and greater than or equal to two.
\end{enumerate}
Note that if the codimension of $N$ is strictly greater than $2$, then 
 $\pi_1(F) = 0$ by Theorem \ref{ES1}.  Moreover, when $T''\subsetneq T'$,  $N$ is fixed by a subtorus $T'''\subset T'$, where $T'/T''\cong T'''$.  Note that  in Case 2.a, while $N$ is fixed by a subtorus $T'''$, it is {\em not} the fixed point set component of $T'''$.

For   Case $2.a$, letting $r=\dim(T''')=dim(T')-\dim(T'')$, 
we prove the following proposition. 
\begin{proposition}\label{case2a} Let $M\in \mathcal{M}_0^{T^k}(n)$ be simply connected and let the $T^k$-action be strictly almost isotropy-maximal. Let $C$ 
be the circle subgroup of $T^k$ fixing $F$.   Suppose that $T^r$ fixes $N$ and $\codim(N)=2r+1$, with $1\leq r\leq k$. Then $C$ acts freely on $N$, and $N/C$ is a closed, simply connected, non-negatively curved manifold with an induced isometric, strictly almost isotropy-maximal $T^{k-r-1}$-action.
\end{proposition}

\begin{proof}
Suppose that $C$ does not act freely on $N$. Let $\Gamma$ be a subgroup of $C$ fixing a point $z\in N$.   Since $\pi_N: E \longrightarrow N$ is an $S^{2r}$-bundle over $N$ and $C$ acts freely on $E$, it follows that $\Gamma$ must act freely on $\pi_N^{-1}(z) \simeq S^{2r}$.  But $\Gamma$ is a subgroup of $C$, hence it acts by orientation preserving isometries. Since the fiber is  $S^{2r}$,  $\Gamma$ must fix a point, giving us a contradiction. 

Moreover, whenever $\pi_1(F)=0$, regardless of the codimension of $N$, the fundamental group of $N$ will be generated by the $C$-orbit in $N$.  This follows exactly as in the proof of Lemma 6.3 in \cite{Wie} (cf. the proof of Part $5$ of Theorem 3.3 in \cite{Do}).
 Since the codimension of $N$ is greater than or equal to $3$ in this case, $\pi_1(F)$ is trivial, and so the closed, non-negatively curved manifold,  $N/C$, must have trivial fundamental group. Finally,  since $T(y)$
 is the smallest possible orbit in $N$, it follows that its image in $N/C$ is also the smallest possible orbit in $N/C$.  Since $C$ acts freely on $N$, $N/C$ is a manifold. Recall that while $N$ is fixed by $T^r$, it is not a fixed point set component of $T^r$.
 Now, while we cannot say that $N$ is non-negatively curved, we do know that $N/C$ in $M/C$ is non-negatively curved  (see the proof of Theorem 3.28 in \cite{Spi} (Theorem \ref{Spindeler} here)). The result then follows.
\end{proof}

   \medskip
 
For Case $2.b$, we prove the following lemma. 
 
  \begin{lemma}\label{Ncase2b}    Let $(M, T^k)$ be a Property C manifold with $M$ simply connected.  Suppose that  $\dim(N)=n-2r-2$ and $N$  is fixed by $T^r\subset T'\subset T$, with $r\geq 0$. Then the induced action on $N$ is locally standard  and $N/T$ and all of its faces are diffeomorphic to disks, after smoothing the corners.
 \end{lemma}

\begin{proof}

Recall that $T'\times H_0$ is fixing the $T(y)$-orbit in $N$ and $\dim(T(y))=2k-n+2$. Let $C$ be the circle fixing $F$.
We first claim that  $H_0$ is trivial.

Suppose  that $H_0$ is non-trivial in order to derive a contradiction. Note that the $(T'\times H_0)/T^r$-action on the unit normal $S^{2(n-k-r-2)-1}$ to the $T(y)$-orbit in $N$ is of rank $n-k-r-2$, that is, of maximal symmetry rank. This implies that the induced $(T'\times H_0)/T^r$-action is $H'$-ineffective, where $H'\cong H_0$ is some non-trivial subgroup of $T'\times H_0$.  
Hence $N$ is fixed  by some non-trivial, disconnected subgroup of $T^k$, and  so $(M, T^k)$ is not a Property C manifold, a contradiction.

We now 
assume that $H_0$ is trivial.  A straightforward calculation shows that the induced torus action on $N$ is isotropy-maximal.  We break the proof into two cases: $r\geq 1$ and $r=0$. 

 In the first case, if $N$ is simply connected, we may use Proposition \ref{lstd} to see that the induced $T^{k-r}$-action is locally standard and  $N/T$ and all of its faces are diffeomorphic to disks, and if $N$ is not simply connected, we use Proposition \ref{lstdcyclic} to obtain the same result. 

Suppose then that $r=0$. Then $\dim(N)=n-2$ and $N$ is not fixed by any isometry. Again, while we cannot say that $N$ is non-negatively curved, we do have  that $N/C$ in $M/C$ is a non-negatively curved manifold.  
Moreover, the induced torus action on $N/C$ is isotropy-maximal. Since $\dim(N)=n-2$, we cannot assume that $\pi_1(F)$ is trivial. We then have two cases: where $N/C$ is simply connected and where $\pi_1(N/C)$ is cyclic.   In both cases, we see that the torus action is locally standard and its quotient space and all of the faces of its quotient space are disks, by Proposition \ref{lstd} for the first case and by Proposition \ref{lstdcyclic} for the second case.  Since $N$ is a principal circle bundle over $N/C$, it follows that the action on $N$ is locally standard and $N/T$ and all of its faces are disks. 

\end{proof}

\begin{remark} In the case where $H_0$ is non-trivial, while $(M, T^k)$ is not a Property C manifold, we still obtain that the action on $N$ is locally standard and  $N/T$ and all of its faces are diffeomorphic to disks.
This follows because $N$ is fixed by a disconnected subgroup of the torus and so $N$ is non-negatively curved, has cyclic fundamental group, and admits an induced almost-effective isotropy-maximal $T^{k-r}$-action. The same argument as in the case where $r\geq 1$ and $H_0$ is trivial then applies.
\end{remark}

We are now ready to prove Theorem \ref{p2}.
\begin{proof}[Proof of Theorem \ref{p2}\label{Proof3.9}]

  Assume first that we are in \hyperref[case1]{Case $1$}.  
Then $T(y)$ is an orbit of smallest dimension.   
Since  $T(y)\subset N$ is a smallest orbit, and $N$ is an invariant submanifold, there is some generalized characteristic submanifold, $F'=\Fix(M; T^1)$,  containing $N$, for some $T^1\subset T^k$.  Indeed, $N$ is the fixed point set component of some subtorus of $T_y$ 
and is one of the $F'_i$ in the chain of inclusions of totally geodesic submanifolds in $F'$, where $i$ denotes the rank of  the subtorus fixing $N$.  By Lemma \ref{F_i}, the induced action of the torus on $N$ is strictly almost isotropy maximal and of cohomogeneity strictly less than $n-k$. 

Note that $(F', T^{k}/T^1)$ is a Property C manifold.
Applying the arguments made earlier in the proof of Theorem \ref{p1} for $F$ to $F'$, we see that the action on $F'$ and hence on $N$ is locally standard. Thus, $(N, T/T^{'''})$ is a Property C manifold.
Since  $N$ has cyclic fundamental group,  Proposition \ref{Ftilde}  gives us the proof of  \hyperref[case1]{Case $1$}. 

  We now consider \hyperref[case2]{Case $2$}, beginning with Case $2.a$.
In order to show that the induced torus action on $N$ is locally standard and $N/T$ is diffeomorphic to a disk,  note that the lifted action on the total space of a principal $S^1$-bundle over a manifold admitting a locally standard torus action is also locally standard. Since $N$ is of dimension $n-2r-1$, the cohomogeneity of the induced $T$-action on $N$ and hence the cohomogeneity of the induced torus action on $N/C$ are both  strictly less than $n-k$. The result then follows  by Proposition \ref{case2a} and the hypothesis.
  For Case $2.b$, Lemma \ref{Ncase2b} gives us the result.
  This completes the proof of  \hyperref[case2]{Case $2$}. 
\end{proof}

\subsection{Step 4 of the Proof of Theorem \ref{6.1analog}}

 In order to finish the proof of Theorem \ref{6.1analog}, it remains to prove Parts 1 and 2, that is, we need to show that the torus action on $M$ is 
locally standard and $M/T$ is diffeomorphic  to a standard disk $D^{n-k}$ after smoothing the corners.

\begin{proof}[Proof of Theorem \ref{6.1analog}]
Recall that by Theorems \ref{p1} and \ref{p2}, the induced torus actions on $F$ and on $N$ are locally standard and both $F/T$  and $N/T$ are diffeomorphic to either a disk or a product of a circle and a disk, after smoothing the corners.

We now show that the $T^k$-action on $D(F)$ is locally standard and that $D(F)/T$ is either a disk or a product of a disk with a circle.
 Since  $F$ is fixed by a circle subgroup, and the action of the circle on the unit normal circle to $F$ is free, the torus action on $D(F)$  is locally standard, as desired.
 
 When $F/T$ is a disk, we have 
 $$D(F)/T \simeq F/T\times I \simeq D^{n-1-k}\times I\simeq D^{n-k}. $$  In the case where $F/T$ is the product of a disk with a circle, we have that  $D(F)/T$ is a $1$-disk bundle over $D^{n-k-2}\times S^1$.
 Recall that a $1$-disk bundle comes from a line bundle and oriented line bundles are in one-to-one correspondence with $H^1(B; \zzz_2)$, where $B$ denotes the base of the line bundle.   Here, the base is $D^{n-k-2}\times S^1$, so there are only two such bundles:   $D^{n-k-1}\times S^1$ or $D^{n-k-2}\times \mathrm{Mb}$, where $\mathrm{Mb}$ is the M\"obius band.

We next show that the $T$-action on $D(N)$ is locally standard and that $D(N)/T$ is diffeomorphic to one of $D^{n-k}$, $D^{n-k-1}\times S^1$, or $D^{n-k-2}\times \mathrm{Mb}$. We break the proof into the same cases as in the proof of Theorem \ref{p2}.

In \hyperref[case1]{Case $1$}, since  $N$ is fixed by some subtorus, and the action of the subtorus on the unit normal sphere to $N$ is of maximal symmetry rank,  the torus action on $D(N)$  is locally standard and the normal disk bundle of $N$ modulo the subtorus action is a disk. 
In the case where $N/T$ is a disk, we see that $$D(N)/T \simeq N/T\times \Delta^l \simeq D^{n-k}.$$
If instead $N/T$ is  a product of a disk with a circle, as we saw above with $D(F)/T$, $D(N)/T$ is then either $D^{n-k-1}\times S^1$ or $D^{n-k-2}\times \mathrm{Mb}$.

We now consider \hyperref[case2]{Case $2$}.  Since $(M, T^k)$ is a Property C manifold, $H_0\subset C$ is trivial by the proof of Theorem \ref{Ncase2b}.  
This fact combined with the invariant neighborhood description given in Display \eqref{e3},   gives us that the $T$-action on $D(N)$  is locally standard.  Moreover, in Case $2.a$, we obtain
$$D(N)/T\cong N/T \times (\Sigma^m\times I)\cong D^{n-k},$$
 and in Case $2.b$,   we obtain
$$D(N)/T\cong N/T \times (\Delta^m\times D^2)\cong D^{n-k}.$$

Since $E$ is a principal circle bundle over $F$, we see that the action on $E$ is locally standard and  since $E/T=F/T$, we see that $E/T$ is also diffeomorphic to a disk or the product of a disk and circle. It then follows that the $T$-action on $M$ is locally standard.

We have seen that $D(F)/T$ and $D(N)/T$ can both be one of $D^{n-k}$, $D^{n-k-1}\times S^1$, or $D^{n-k-2}\times \mathrm{Mb}$. Recall that by Corollary 6.3 of Chapter 2 in \cite{Br}, $\pi_1(M/T)$ is trivial. 
Hence, if both $D(F)/T$ and $D(N)/T$ are diffeomorphic to either $D^{n-k-1}\times S^1$ or $D^{n-k-2}\times \mathrm{Mb}$, then $M/T$ is diffeomorphic to $D^{n-k-1}\times S^1$ or $D^{n-k-2}\times \mathrm{Mb}$, giving us a contradiction. 
If instead one of $D(F)/T$ or $D(N)/T$ is diffeomorphic to a disk and the other to $D^{n-k-2}\times \mathrm{Mb}$, then $M/T$ is diffeomorphic to $D^{n-k-2}\times \rrr\mathrm{P}^2$, again giving us a contradiction. Likewise, one sees that if one of $D(F)/T$ or $D(N)/T$ is diffeomorphic to $D^{n-k-1}\times S^1$ 
and the other to $D^{n-k-2}\times \mathrm{Mb}$, then $M/T$ is homotopy equivalent to $D^{n-k-2}\times \mathrm{Mb}$, a contradiction.
There are two cases left to consider, either both $D(F)/T$ or $D(N)/T$ are diffeomorphic to a disk, or one of $D(F)/T$ or $D(N)/T$ is diffeomorphic to a disk and the other to $D^{n-k-1}\times S^1$.
In the first case it is clear that $M/T$ is diffeomorphic to a disk. In the latter case, considering the quotient of the disk bundle decomposition, one sees that 
$M/T$ is diffeomorphic to a disk, as desired. 
\end{proof}

The proof of Theorem \ref{6.1analog} is now complete.

\section{The Proof of Theorem \ref{extend}}\label{s4}

The goal of this section is to prove Theorem \ref{extend}.
Recall that by Theorem \ref{Spindeler}, $M$ decomposes as $$M=D(F)\cup_E D(N),$$
where $F$ is a generalized characteristic submanifold fixed by a circle subgroup $C$ of $T^k$, $N$ is a $T^{k}$-invariant submanifold at maximal distance from $F$, and $E$ is 
the common boundary of their respective disk bundles. Moreover, by Corollary \ref{Ncyclic}, both $F$ and $N$ have cyclic fundamental group.

We first show that the extension, if it exists, is unique.
\begin{proposition}\label{unique} Let $T^k$ act smoothly and (almost) effectively  on $M^n$, a connected manifold, and assume that the action is strictly almost isotropy-maximal. If the $T^k$-action admits a non-trivial extension to a smooth, (almost) effective, and isotropy-maximal $T^{k+1}$-action, then the extension is unique. 
\end{proposition}

\begin{proof} Suppose instead that there exist two different (almost) effective $T^{k+1}$ extensions of the $T^k$-action, each of which is isotropy-maximal.  Let each be denoted by $T^{k}\times T^1_i$, $i=1, 2$, with $T^1_1$ and $T^1_2$ distinct, smooth, and (almost) effective circle actions on $M^n$. 
Consider the group $G$ generated by $T^k$, $T^1_1$, and $T^1_2$ and the subgroup $H=T^1_1\cap T^1_2$. Then $H$ is isomorphic to the identity, $\zzz_k$, or $S^1$. In the first two cases, $G$ must be of rank $k+2$. But, by Lemma \ref{ishida},  since  each of the $T^{k+1}_i$ actions is the largest possible (almost) effective torus action on $M$, we obtain a contradiction.  Hence $T^1_1\cap T^1_2\cong S^1$ and therefore $T^1_1=T_2^1$ and the result follows.
\end{proof}

\bigskip

\begin{proof}[Proof of Theorem  \ref{extend}] The proof is again by induction on the cohomogeneity of the action.  As described in the proof of Theorem \ref{6.1analog}, the anchor of the induction consists of an isometric, strictly almost isotropy-maximal $T^1$-action on a non-negatively curved $3$-manifold. We already saw in the proof of Theorem \ref{6.1analog} that this $T^1$-action is locally standard and $T^1$-fixed-point-homogeneous. In particular, $M^3$ decomposes as a union of disk bundles,  $M^3=D(F)\cup D(N)$, where $F$ is a fixed circle and $N$ is a circle orbit. Moreover, $M\,\setminus\, F$ consists only of principal orbits of the $T^1$-action. Thus, we can apply the construction contained in the proof of Proposition 3.6 in \cite{GGSpi} to obtain a smooth $T^2$-action extending the isometric $T^1$-action. By construction, the $T^2$-action is locally standard. Note that the extended action is by cohomogeneity one and since $M^3$ is closed and simply connected, it follows that the quotient space is an interval. Uniqueness of the extension follows directly from Proposition \ref{unique}. We summarize this result in the following lemma.

\begin{lemma}\label{GGSP} Let $M^3$ be a closed, simply connected, non-negatively curved Riemannian manifold admitting an isometric, effective, and strictly almost isotropy-maximal $T^1$-action. Then the  isometric $T^1$-action on $M^3$ may be extended to a unique, smooth, isotropy-maximal $T^2$-action. Moreover, the extended $T^2$-action is locally standard and $M^3/T^2$ is a $1$-disk and its faces are $0$-disks. 
\end{lemma}

We now consider an isometric, strictly almost isotropy-maximal $T^k$-action on $M^n$, a closed, simply connected, non-negatively curved Riemannian manifold.
For any isometric, strictly almost isotropy-maximal $T^l$-action on a closed, simply connected, non-negatively curved manifold, $M^m$,  of cohomogeneity $m-l<n-k$, we  assume that  the isometric $T^l$-action extends to a smooth, locally standard, and isotropy-maximal $T^{l+1}$-action.

We break the proof into 
two steps: Step $1$, where we extend the torus action first  on $D(F)$, and Step $2$, where we prove the corresponding result for  $D(N)$. In order to complete the proof of Theorem \ref{extend}, we note that by Proposition \ref{unique}, since the extended (almost) effective actions agree on their common boundary, $E$,  we obtain a smooth, isotropy-maximal, (almost) effective $T^{k+1}$-action on all of $M$. Moreover, by construction, the $T^{k+1}$-action is  locally standard. We observe that in the case where the extended action is almost effective, we may then mod out by the finite ineffective kernel to obtain the desired effective, smooth, locally standard, isotropy-maximal $T^{k+1}$-action on all of $M$.

It remains only to prove Steps 1 and 2. 

\medskip
\noindent
{\bf Step 1}:  
The $T^k$-action on  $D(F)$ extends to a smooth, (almost) effective, isotropy maximal and locally standard $T^{k+1}$-action, and $F/T$ and $D(F)/T$ and all of their faces are diffeomorphic to disks, after smoothing the corners.

\vspace{2mm}

We first show that we can extend the induced isometric, strictly almost isotropy-maximal $T^{k-1}$-action on $F$ to a smooth, effective, isotropy-maximal, and locally standard $T^k$-action.

We first consider the case where $\pi_1(F)$ is finite. By Part 1 of Lemma \ref{Fpi1nontrivial}, we may lift the action to  a strictly almost isotropy-maximal torus action on $\widetilde{F}$, a closed, simply-connected, non-negatively curved manifold. Theorem \ref{lift} guarantees us that the lift commutes with the deck transformations. By the induction hypothesis, we may extend the $T^{k-1}$-action on $\widetilde{F}$ to a smooth,  isotropy-maximal and locally standard $T^{k}$-action. This then allows us to extend the $T^{k-1}$-action on $F$ to a smooth, isotropy-maximal and locally standard $T^{k}$-action.  We note in this case that the action on $F$ is almost effective.

Likewise, if $\pi_1(F)\cong \zzz$, then by Part 2 of Lemma \ref{Fpi1nontrivial}, we have two cases according to whether the lifted group is compact or not. 
In the case of $\widetilde T^{k-1} = T^{k-1}$, the action of $T^{k-1}$ on
 $\widetilde{F} \simeq \bar{F} \times \rrr$ is isotropy-maximal   
on the $\bar{F}$ factor and trivial on the $\rrr$-factor. 
We can then extend the $T^{k-1}$-action on  $\widetilde{F}$ by an $\rrr$, obtaining a smooth, isotropy-maximal $(T^{k-1} \times \rrr)$-action on $\bar{F}\times \rrr$. 
For the case of $\widetilde T^{k-1} = T^{k-2} \times \rrr$, the action is strictly almost isotropy-maximal on the $\bar{F}$ factor and trivial on the $\rrr$-factor.  We can then use the induction hypothesis to extend the $T^{k-2}$-action on $\bar F$ to again obtain a smooth, isotropy-maximal $(T^{k-1} \times \rrr)$-action on $\bar{F}\times \rrr$.  
In both cases this action commutes with the deck transformations and hence induces a smooth (almost) effective $T^k$-action on $F$ that is both isotropy-maximal and locally standard. 

In both of the above cases, by Proposition \ref{lstd}, we see that $\widetilde{F}/T^{k}$ and all of its faces are diffeomorphic to disks, after smoothing the corners.
A similar argument as in the proof of Theorem \ref{Ftilde} then gives us that $F/T^k$ and all of its faces are diffeomorphic to disks, after smoothing the corners.

Note that the  extended $T^k$-action on $F$ commutes with the $C$-action on $F$, since the latter action is trivial.  To see that the action now extends to $D(F)$, note that the $C$-action acts on the normal space to $F$ by rotating the fibers.
As in the proof of Proposition 3.6 \cite{GGSpi}, we see that this gives us the desired (almost) effective $T^{k+1}$-action on $\nu(F)$ and hence on $D(F)$ via the exponential map.  
Since $F/T^k$  
and all of its faces are diffeomorphic to disks, after smoothing the corners, we may argue as in Step 4 of the proof of Theorem  \ref{6.1analog} to see that $D(F)/T$ and all of its faces are diffeomorphic to disks, after smoothing the corners.

\medskip
\noindent
{\bf Step 2}:  
The $T^k$-action on  $D(N)$ extends to a smooth, (almost) effective, isotropy maximal and locally standard $T^{k+1}$-action, and $F/T$ and $D(F)/T$ and all of their faces are diffeomorphic to disks, after smoothing the corners.

\medskip

As in the proof of Theorem \ref{6.1analog} there are two cases to consider.  Recall that $\pi_F: E \rightarrow F$ is a $C$-bundle over $F$. Recall also that we choose $x \in F$ such that $T(x)$ is a smallest orbit in $F$ and  $x' \in E$ such that $x'\in \pi_F^{-1}(x)$. We let $T(y)$ be the projection of the orbit $T(x')$ to $N$.  There are then two cases: 
\begin{enumerate}
\item[] {\bf Case ${\bf 1}$:}  $T(y)$ is an orbit of smallest dimension, that is, $\dim(T(y)) = \dim(T(x))$; or
\item[] {\bf Case ${\bf 2}$:}  $T(y)$ is not an orbit of smallest dimension, that is, $\dim(T(y))=\dim(T(x))+1$.
\end{enumerate}
As before, Case 2 breaks into two further subcases: 
 \begin{enumerate}
 \item[] {\bf Case ${\bf 2.a}$:} The codimension of $N$ is odd and greater than or equal to $3$; and 
\item[] {\bf Case ${\bf 2.b}$:}  The codimension of $N$ is even and greater than or equal to $2$.
\end{enumerate}
Recall that in Case 2.a,  $\codim(N)=2r+1$, $r\geq 1$, and  in Case 2.b, $\codim(N)=2r+2$, $r\geq 0$,  
where $N$ is fixed by a subtorus of rank $r$.

For Case 1, recall that as we saw in the \hyperref[Proof3.9]{proof of Case 1 of Theorem \ref{p2}}, $N$ is contained in some generalized characteristic submanifold, $F'$, containing $T(y)$. In fact, $N$ is the fixed point set component of some subtorus $T^i\subset T_y$, with $1\leq i\leq n-k-1$, and is thus one of the $F'_i$ in the chain of inclusions of totally geodesic submanifolds in $F'$. 
Moreover, $T^i$ acts by maximal symmetry rank on the unit normal sphere to $N$.
Since $N$ has cyclic fundamental group,  as in the proof of 
Step 1, we may now extend the induced torus action on $N$ via the induction hypothesis. Since  the $T^i$-action on the unit normal sphere to $N$ is of maximal symmetry rank, and $T^i$ commutes with the circle of the extension, we may extend the action to $D(N)$. Moreover, $N/T$, $D(N)/T$, and all of their respective faces are diffeomorphic to disks, after smoothing the corners.

For  Case $2.a$,  the arguments in 
 Proposition \ref{case2a} do not rely on $N$ satisfying \hyperref[condC]{Property C}. So, we may conclude here as well that $N$ is a principal $C$-bundle over a closed, simply connected, non-negatively curved manifold, $N/C$, with an induced isometric strictly almost isotropy-maximal $T^{k-r-1}$-action.  By the induction hypothesis, we may extend the isometric $T^{k-r-1}$-action on $N/C$ to a $T^{k-r}$-action on $N/C$ that is smooth, isotropy-maximal and locally standard. By Theorem \ref{Su},  we may lift the $T^{k-r}$-action  to $N$ to obtain a smooth, isotropy-maximal, and locally standard $T^{k-r+1}$-action on $N$. Further, the $T^r$-action on the unit normal sphere to $N$ is of maximal symmetry rank and commutes with the circle of the extension.
We then argue as in Case $1$   to see that there is a smooth, (almost) effective, isotropy-maximal, and locally standard $T^{k+1}$-action on $D(N)$ and that $N/T$, $D(N)/T$, and all of their respective faces are diffeomorphic to disks, after smoothing the corners.

For Case $2.b$, we first consider the case where  $N$ satisfies \hyperref[condC]{Property C}. There are two subcases to consider: when   $r\geq 1$ and $r=0$. 

We begin with the case when $r\geq 1$ and  $N$ satisfies \hyperref[condC]{Property C}. As we saw in the proof of Theorem \ref{6.1analog}, 
 the induced torus action on $N$ is  isotropy-maximal and locally standard. The action of $T^r$ on its normal $S^{2r+1}$  is strictly almost isotropy-maximal. The $T^r$-action  fixes $N$, so we extend the $T^r$-ineffective $T^k$-action on $N$ to a $T^{r+1}$-ineffective $T^{k+1}$-action on $N$. By the induction hypothesis, we may also extend the effective $T^r$-action on the normal $S^{2r+1}$ to an effective $T^{r+1}$-action.  By exponentiating,
we then have a smooth, (almost) effective, isotropy-maximal, and locally standard $T^{k+1}$-action on all of $D(N)$.  We use the same argument as in Case 1 to see that $N/T$, $D(N)/T$, and all of their respective faces are diffeomorphic to disks, after smoothing the corners.

Now, consider the case when $r=0$ and  $N$ satisfies \hyperref[condC]{Property C}. In particular,  $N$ is fixed by no subgroup of $T^{k}$. Then we extend the action on $D(N)$ as in the proof of Proposition 3.6 in \cite{GGSpi}. That is, we let the additional $T^1$ fix $N$ and act freely on the unit normal circle to $N$. This gives us the desired smooth, (almost) effective, isotropy-maximal, and locally standard $T^{k+1}$-action on $D(N)$.  In fact, in the proof of Lemma \ref{Ncase2b}, we saw that $N/T^k$ and all of its faces are diffeomorphic to disks, after smoothing the corners and since the $T^1$ extending the $T^{k}$-action fixes $N$, we see that $N/T^{k+1}$ has exactly the same quotient space.  Since the $T^1$ fixing $N$ acts freely on the unit normal circle to $N$, $D(N)/T$ and all of its faces are diffeomorphic to disks, after smoothing the corners.

If $N$ does not satisfy \hyperref[condC]{Property C},  then we extend the $T^k$ action as we did in Case 2b when $r\geq 1$. That is, the $(T^r\times H')$-action in the proof of Proposition \ref{Ncase2b} fixes $N$, so we extend the $(T^r\times H')$-ineffective $T^k$-action on $N$ to a $(T^{r+1}\times H')$-ineffective $T^{k+1}$-action on $N$. By the induction hypothesis, we may also extend the linear $(T^r\times H')$-action on the normal $S^{2r+1}$ to a linear $(T^{r+1}\times H')$-action which is $H''\cong H'$-ineffective, where $H''\subset T^{r+1}\times H'$.  By exponentiating,
we then have a smooth $T^{k+1}$-action on all of $D(N)$ that is  $H''$-ineffective. We then obtain an almost effective, smooth, isotropy-maximal, locally standard $T^{k+1}$-action on $D(N)$.   Since $N$ is non-negatively curved, we use once again the same argument as in Case 1 to see that that $N/T$, $D(N)/T$, and all of their respective faces are diffeomorphic to disks, after smoothing the corners.
\end{proof}


\section{The Proofs of the Theorems \ref{RE}, \ref{RE2}, and \ref{main}}\label{s5}

In this section we prove Theorems \ref{RE}, \ref{RE2}, and \ref{main}. 
We observe that the proofs of   Theorems \ref{RE} and \ref{RE2} are quite similar and so are done simultaneously.

\begin{proof}[{\bf Proof of Theorems \ref{RE} and  \ref{RE2}}] For an isotropy-maximal $T^k$-action, we claim that the free rank is at most $2k-n$. If we suppose instead that the free rank is strictly less than $2k-n$, then we can apply the arguments of the proof of Proposition 5.1 in \cite{ES2}, to see that there must be a point in $M$ fixed by a subtorus of rank strictly greater than $n-k$. However, this is impossible, as it would contradict the fact that the maximal rank of an isotropy subgroup of a $T^k$-action on an $n$-manifold is $n-k$. 

Let $T^{2k-n}$ be the almost freely acting subtorus of $T^{k}$. Then $M^{n}/T^{2k-n}$ is a closed, simply connected, $(2n-2k)$-dimensional orbifold admitting an induced smooth $T^{n-k}$-action with a $T^{n-k}$ fixed point, that is, a torus orbifold.  Since the quotient of a closed, rationally elliptic manifold by an effective,  almost free torus action is rationally elliptic, we have that $X^{2n-2k}=M^{n}/T^{2k-n}$ is rationally elliptic (see Observation 6.6 in \cite{ES2}). Letting $P^{n-k}$ be as in Display \eqref{P^n}, we see that in the case of Theorem \ref{RE},  it follows that $M^{n}/T^{k}=X^{2n-2k}/T^{n-k}$ is homeomorphic to $P^{n-k}$ by Proposition \ref{P}. In the case of Theorem \ref{RE2}, again by Proposition \ref{P}, we have that $M^{n}/T^{k}$ is diffeomorphic to $P^{n-k}$, after smoothing the corners. 
Note that the homeomorphism or diffeomorphism is weight-preserving by construction, see \cite{Wie}.

By assumption, the $T^k$-action on $M$ is locally standard. It then follows by the  Cross-Sectioning Theorem in \cite{ES2} that a cross-section for the action on $M$ exists.  In the case of 
Theorem \ref{RE}, the weight-preserving homeomorphism between the orbit spaces then yields an equivariant homeomorphism between the total spaces  by Theorem 3.5 in \cite{ES2}.
  For Theorem \ref{RE2},  the weight-preserving diffeomorphism between the orbit spaces  
  gives rise to an equivariant diffeomorphism 
between the total spaces by Theorem 3.7 in \cite{ES2}, taking into account Remark \ref{Wie}.  We thus obtain the desired total spaces as in the proof of Case $2$ of Theorem \ref{thma}.
\end{proof}

We are now ready to prove Theorem \ref{main}. 
\begin{proof}[{\bf Proof of Theorem \ref{main}}]
As mentioned in the Introduction, Theorem \ref{thma} takes care of the case when the torus action is isotropy-maximal, so we restrict our attention to the case where the torus action is strictly almost isotropy-maximal.
We first note that  $M$ is rationally elliptic by Theorem \ref{misre}. 
We then apply Theorem \ref{extend} to extend the isometric $T^{k}$-action on $M^{n}$ to 
a smooth, isotropy-maximal, and locally standard $T^{k+1}$-action. Moreover, Theorem \ref{extend} gives us that $M/T^{k+1}$ and all of its faces, and in particular, its $4$-dimensional faces, are diffeomorphic to disks, after smoothing the corners. The result now follows by Theorem \ref{RE2}.
\end{proof}


\end{document}